\documentclass[12pt]{amsart}
\usepackage[left=3cm,right=3cm,top=2cm,bottom=2cm]{geometry} 
\usepackage{hyperref}
\usepackage{amsmath}
\usepackage{amssymb}
\usepackage{amsthm}
\usepackage[parfill]{parskip}
\usepackage{tikz-cd} 
\usepackage{mathtools}
\usepackage{faktor}

\newtheorem{theorem}{Theorem}[section]
\newtheorem{lemma}{Lemma}[section]
\newtheorem{definition}{Definition}[section]
\newtheorem{proposition}{Proposition}[section]
\newtheorem{remark}{Remark}[section]
\newtheorem{corollary}{Corollary}[section]

\newcommand{\R}{\mathbb{R}}
\newcommand{\Z}{\mathbb{Z}}
\newcommand{\C}{\mathbb{C}}
\newcommand{\Q}{\mathbb{Q}}
\newcommand{\HH}{\mathbb{H}}
\newcommand{\F}{\mathcal{F}}
\newcommand{\OO}{\mathcal{O}}

\newcommand{\Nc}{\mathcal{N}}
\renewcommand{\Mc}{\mathcal{M}}
\newcommand{\Zc}{\mathcal{Z}}
\newcommand{\Vc}{\mathcal{V}}

\newcommand{\J}{\mathcal{J}}

\newcommand{\rvline}
{\hspace*{-\arraycolsep}
\vline\hspace*{-\arraycolsep}}

\begin{document}

\title{A generalization of Inoue surfaces $S^+$}
\author{David Petcu}
\address{
		\textsc{\indent University of Bucharest, Faculty of Mathematics and Computer Science\newline 
			\indent 14 Academiei Str., Bucharest, Romania \newline
			\indent \indent and \newline
			\indent Institute of Mathematics ``Simion Stoilow'' of the Romanian Academy\newline 
			\indent 21 Calea Grivitei Street, 010702, Bucharest, Romania}}

\email{david.petcu@my.fmi.unibuc.ro}
\subjclass[2000]{53C55, 22E25, 32J18}
\keywords{Solvmanifolds,  non-Kähler manifolds}
\date{}

\begin{abstract}
Using Lie groups with left-invariant complex structure, we construct new examples of compact complex manifolds with flat affine structure in arbitrarly high dimensions. In the 2-dimensional case, we retrieve the Inoue surfaces $S^+$.
\end{abstract}

\maketitle

\section{Introduction}

The surfaces introduced by Inoue in 1974 are complex, compact, non-Kähler manifolds. Alongside Hopf surfaces, Kodaira surfaces, and tori, they serve as examples of surfaces that possess a natural flat affine structure. This means that there exists a covering of the manifold with coordinate charts such that all the transition maps between these charts are affine functions.\par
To construct manifolds with this property, one can begin with a real Lie group that is endowed with a left-invariant complex structure \cite{wall}. By taking the quotient of this Lie group by the left action of a discrete cocompact subgroup, one obtains the desired manifold. In this paper, we utilize this approach to introduce a higher-dimensional generalization of Inoue surfaces of type $S^+$.\par
There have been several constructions proposed to date that generalize different types of Inoue surfaces. For instance, Oeljeklaus-Toma (\cite{OeTo})  and Endo-Pajitnov (\cite{EPaj}) manifolds generalize $S^0$ surfaces, while Oeljeklaus-Miebach manifolds (\cite{OeMi}) generalize $S^+$ surfaces. Despite the differences in their specific constructions, all these manifolds share the fundamental property of admitting an affine structure, as defined above.\par
By expanding on Inoue's original work, our paper aims to further explore the geometric and topological properties of these higher-dimensional analogues. The main tool we use to achieve this is group cohomology. 
 Specifically, we use it to prove that the proposed manifolds do not contain hypersurfaces. 
We also investigate the natural generalization of the left-invariant metric structure that exists on the $S^+$ surfaces.

\hfill

{\bf Acknowledgements.} The author would like to thank the anonymous referee for valuable suggestions for improving the text and for corrections to the first version. We also thank Victor Vuletescu for pointing the problem and Liviu Ornea for useful remarks. The present work is partly supported by the PNRR-III-C9-2023-I8 grant CF 149/31.07.2023 {\em Conformal Aspects of Geometry and Dynamics}. 

\section{Construction}

Let $d=2m+1$ be a positive integer and let $G$ be the Lie group:

\begin{equation}\label{elgen}
G=\left\{
\begin{pmatrix}
I_d & \rvline & \begin{matrix} b_d \\ \vdots \\ b_1 \end{matrix} & \rvline & 
\begin{matrix}
c_{1d} & \dots & c_{dd}\\
\vdots & \ddots & \vdots\\
c_{11} & \dots & c_{d1}\\
\end{matrix}\\
\hline
0 & \rvline & \alpha & \rvline & 
\begin{matrix}
a_1 & \dots & a_d
\end{matrix}\\
\hline
0 & \rvline & 0 & \rvline& I_d
\end{pmatrix}
\vert  \ \alpha \in \R_{>}; a_i, b_j, c_{ij} \in \R
\right\}
\end{equation}
Let  $n=\frac{(d+1)^2}{2}-1$, there is a diffeomorphism from $G$ to $\HH \times \C^n$ (as real manifolds):
$$\begin{pmatrix}
I_d & \rvline & \begin{matrix} y_{d0} \\ \vdots \\ y_{10} \end{matrix} & \rvline & 
\begin{matrix}
x_{d0} & y_{d1} & x_{d1} & \dots & y_{dm} & x_{dm}\\
\vdots & \vdots & \vdots & \ & \vdots & \vdots\\
x_{10} & y_{11} & x_{11} & \dots & y_{1m} & x_{1m}\\
\end{matrix}\\
\hline
0 & \rvline & w_2 & \rvline & 
\begin{matrix}
w_1 & y_{01} & x_{01} & \dots & y_{0m} & x_{0m}
\end{matrix}\\
\hline
0 & \rvline & 0 & \rvline& I_d
\end{pmatrix} \to (w_1+iw_2, x_{ij}+iy_{ij})\text{.}$$
 Thus, $G$ can be seen as a complex manifold. Moreover, the left actions by elements of $G$ are holomorphic with respect to this structure. \par

The largest nilpotent subgroup of $G$ is
$$\Nc=\left\{
\begin{pmatrix}
I_d & \rvline & (b_j) & \rvline & (c_{ij})\\
\hline
0 & \rvline & 1 & \rvline & (a_i) \\
\hline
0 & \rvline & 0 & \rvline& I_d
\end{pmatrix} \vert \ a_i, b_j, c_{ij} \in \R
\right\}$$
This gives the short (split) exact sequence of groups:
\begin{equation}\label{sir11}
\begin{tikzcd}
	1 & \Nc & G & {\R_{>}} & 1
	\arrow[from=1-1, to=1-2]
	\arrow[from=1-2, to=1-3]
	\arrow[from=1-3, to=1-4]
	\arrow[from=1-4, to=1-5]
\end{tikzcd}
\end{equation}
Denote by $\Zc$ the first derived group of $\Nc$ (i.e. $\Zc=[\Nc,\Nc]$):
$$\Zc=\left\{
\begin{pmatrix}
I_d & \rvline & 0 & \rvline & (c_{ij})\\
\hline
0 & \rvline & 1 & \rvline & 0 \\
\hline
0 & \rvline & 0 & \rvline& I_d
\end{pmatrix} \vert \ c_{ij} \in \R
\right\}.$$
This subgroup is isomorphic to $\R^{d^2}$.

The following Lemma is immediate, but we insert its statement here as it plays an important role in the sequel:
\begin{lemma}
The subgroup $\Zc$ is the center of $G$.
\end{lemma}

 We get another short exact sequence of groups
\begin{equation}\label{sir2}
\begin{tikzcd}
	1 & \Zc & \Nc & \Vc & 1
	\arrow[from=1-1, to=1-2]
	\arrow[from=1-2, to=1-3]
	\arrow["\pi", from=1-3, to=1-4]
	\arrow[from=1-4, to=1-5]
\end{tikzcd}
\end{equation}
where the quotient Lie group $\Vc$ is isomorphic to $\R^{2d}.$ \par
Next, we will search for a discrete cocompact subgroup $\Lambda$ of $G$. To do this we will need the folllowing result, which is easy to prove:

\begin{lemma} \label{disc}
Let 
\[\begin{tikzcd}
	1 & {\mathcal{L}} & {\mathcal{M}} & {\mathcal{R}} & 1
	\arrow[from=1-1, to=1-2]
	\arrow["i", from=1-2, to=1-3]
	\arrow["\pi", from=1-3, to=1-4]
	\arrow[from=1-4, to=1-5]
\end{tikzcd}\]
be a short exact sequence of Lie groups such that $i$ is a smooth embedding. Let $\Lambda_{\Mc}$ be a subgroup of $\Mc$ and denote by $\Lambda_{\mathcal{L}}$ the intersection $\Lambda_{\Mc} \cap \mathcal{L}$ and by $\Lambda_{\mathcal{R}}$ the image $\pi(\Lambda_{\Mc}) \subset \mathcal{R}$. If $\Lambda_{\mathcal{L}}$ is a discrete subgroup of $\mathcal{L}$ and $\Lambda_{\mathcal{R}}$ is a discrete subgroup of $\mathcal{R}$, then $\Lambda_{\Mc}$ is a discrete subgroup of $\Mc$.
\end{lemma}

 Henceforth, in order to find the subgroup $\Lambda$, we look for a cocompact discrete subgroup $\Lambda_\Nc \subset \Nc$ and a cocompact discrete subgroup $\langle g_0 \rangle$ of $\R_>$ (where $\R_>$ is seen as a subgroup of $G$ via the natural splitting of \ref{sir11}) such that $g_0\Lambda_\Nc g_0^{-1}=\Lambda_\Nc$. Then, the group $\Lambda:=\langle \Lambda_\Nc, g_0 \rangle$ will be  cocompact and discrete in $G$, as desired. Indeed, we get discreteness by applying lemma (\ref{disc}) while for cocompactness we observe that the action $g_0$ descends to an action on the quotient $G/\Lambda_{\Nc}$, which is difeomorphic to a fiber bundle over $\R_>$ with fibers difeomorphic to $\Nc/\Lambda_\Nc$. Then, $G/\Lambda_\Nc$ is difeomorphic to a fiber bundle with fiber $\Nc/\Lambda_\Nc$ over the base $\R_>/\langle g_0 \rangle$.\par

To produce a cocompact lattice $\Lambda_\Nc \subset \Nc$ of $\Nc$, first notice that $D\coloneqq \Lambda_\Nc \cap \Zc$ must be a cocompact subgroup of $\Zc$ (cf. e.g. \cite{rag}, Prop 2.17). Assume $\Lambda_\Nc$ is such a subgroup, let $h_1 \dots h_{d^2}$ be the generators of the cocompact lattice $D$ and denote $\pi(\Lambda_\Nc)$ by $\Lambda_\Vc$. Then, since $\Lambda_\Nc$ is cocompact in $\Nc$  we get that $\Lambda_\Vc  \otimes \R = \Vc$, hence we only have to check that  $\Lambda_\Vc$ is discrete in $\Vc$. Before continuing, we define the bilinear map $[\cdot,\cdot]_\Vc:\Vc \times \Vc \to \Zc$ as $[\lambda, \lambda']_\Vc = [\lambda_\Nc, \lambda'_\Nc]$, where $\lambda_\Nc, \lambda'_\Nc$ are elements in $\Nc$ such that $\pi(\lambda_\Nc)=\lambda$ and $\pi(\lambda'_\Nc)=\lambda'$. This map is indeed well defined, as the group commutator on $\Nc$ depends only on the images through $\pi$. Moreover, it is easy to see that for every nonzero $\lambda \in \Vc$ there exists a $\lambda'$ such that $[\lambda,\lambda']_\Vc \neq 0$. Now, if $\Lambda_\Vc$ is not discrete in $\Vc$, then $[\Lambda_\Vc, \Lambda_\Vc]_\Vc$ is not discrete in $\Zc$ and, as remarked before, this is a contradiction, therefore $\Lambda_\Vc$ is a lattice. Choose $g_1 \dots g_{2d} \in \Vc$ such that their images in $\Vc$ generate the lattice $\Lambda_\Vc$. Thus, we have shown that $\Lambda_\Vc$ must be of the form $\langle g_1 \dots g_{2d}, h_1 \dots h_{d^2} \rangle$, where
$$
 g_k=
\begin{pmatrix}
I_d & \rvline & (b_j^k) & \rvline & (c_{ij}^k)\\
\hline
0 & \rvline & 1 & \rvline & (a_i^k) \\
\hline
0 & \rvline & 0 & \rvline& I_d
\end{pmatrix}\text{, }k\in \{1 \dots 2d\}\ \ \ \ h_k=
\begin{pmatrix}
I_d & \rvline & 0 & \rvline & (d_{ij}^k)\\
\hline
0 & \rvline & 1 & \rvline & 0 \\
\hline
0 & \rvline & 0 & \rvline& I_d
\end{pmatrix}\text{, }k\in \{1 \dots d^2\}
$$

Eventually, recall that the generator $g_0$ was chosen to be  of the form

\begin{equation}\label{g0}
g_0=
\begin{pmatrix}
I_d & \rvline & 0 & \rvline & 0\\
\hline
0 & \rvline & \alpha & \rvline & 0 \\
\hline
0 & \rvline & 0 & \rvline& I_d
\end{pmatrix}
\end{equation}
Imposing that $g_0g_kg_0^{-1} \in \Lambda_\Nc$ for all $k$, we get that there exist integers $\{n_{i,j} \vert 1\leq i,j \leq 2d\}$ and $\{p_{i,j}\vert 1\leq i\leq 2d; 1\leq j \leq d^2\}$ such that 
\begin{equation}\label{conj}
g_0g_kg_0^{-1}=g_1^{n_{k,1}} \dots g_{2d}^{n_{k,2d}}h_1^{p_{k,1}} \dots h_{d^2}^{p_{k,d^2}}
\end{equation}
as all elements of $\Lambda_\Nc$ are of this form. This implies that 
$$\alpha a_i^k = \sum_{j=1}^{2d}n_{kj}a_i^j\text{ and }\frac{1}{\alpha}b_i^k = \sum_{j=1}^{2d}n_{kj}b_i^j$$
Thus,  $\langle (a_i^1 \dots a_i^{2d})\vert \ i\in\{1 \dots d\}\rangle$ and $\langle (b_i^1 \dots b_i^{2d})\vert \ i\in\{1 \dots d\}\rangle$ are the eigenspaces of the matrix $N\coloneqq (n_{i,j})$ associated to the eigenvalues $\alpha$ and $\frac{1}{\alpha}$ respectively. Moreover, the matrix $N$ is from $SL_{2d}(\Z)$ and has eigenvalues $\alpha$ and $\frac{1}{\alpha}$, each with multiplicity $d$. 

\begin{remark}
To ensure the existence of the matrix $N$ as described above, we need to choose the value of $\alpha \in \R_>$. The minimal polynomial of $N$ is $P(X)=(X-\alpha)(X-\frac{1}{\alpha})$, so $\alpha$ is a real quadratic integer. Thus, $\alpha$ should be of the form $\frac{\beta\pm\sqrt{\beta^2-4}}{2}$, for some integer $\beta \geq 3$. For any such choice of $\beta$, there are indeed integer matrices $N$ with the required eigenvalues, take for example the block-diagonal matrix formed by repeating the $2\times 2$-block $\begin{pmatrix}
1 & 1 \\
\beta-1 & \beta\\
\end{pmatrix}$. We will assume $\alpha > 1$.

\end{remark}

\begin{lemma}\label{dis}
For any choice of the matrices
$$
 g_k=
\begin{pmatrix}
I_d & \rvline & (b_j^k) & \rvline & (c_{ij}^k)\\
\hline
0 & \rvline & 1 & \rvline & (a_i^k) \\
\hline
0 & \rvline & 0 & \rvline& I_d
\end{pmatrix}\text{, }k\in \{1 \dots 2d\}\
$$
satisfying the relations \ref{conj}, the subgroup of $\Zc$ generated by  all commmutators
$g_ig_jg_i^{-1}g_j^{-1}$ is discrete.
\end{lemma}

\begin{proof} Let $\varphi$ be the nontrivial automorphism of the field $\Q(\alpha)$, that is the automorphism sending $\sqrt{\delta}$ to $-\sqrt{\delta}$, where $\delta$ is the discriminant of $\Q(\alpha)$. We will first prove the lemma for a special choice of the matrices $g_i$: assume the eigenvectors $v_i \coloneqq (a_i^1 \dots a_i^{2d})$ are in $\Q(\alpha)^{2d}$ and that $w_i \coloneqq (b_i^1 \dots b_i^{2d}) = \varphi(a_i^1 \dots a_i^{2d})$ for all $i$. If we set $$g_sg_tg_s^{-1}g_t^{-1} =
\begin{pmatrix}
I_d & \rvline & 0 & \rvline & (e_{ij}^{s,t})\\
\hline
0 & \rvline & 1 & \rvline & 0 \\
\hline
0 & \rvline & 0 & \rvline& I_d
\end{pmatrix}
\text{,}$$
then 
\begin{equation}\label{com}
e_{ij}^{s,t}=b_i^sa_j^t-a_j^sb_i^t
\end{equation}
so $\varphi(e_{ij}^{s,t})=-e_{ji}^{s,t}$. Therefore, $(e_{ij}) \in \langle T_{ij}-T_{ji}, \sqrt{\delta}T_{ij}+\sqrt{\delta}T_{ji} \rangle_{\Q}$, where $T_{ij}$ is the matrix with $1$ in position $ij$ and $0$ everywhere else. Thus $\Zc$ is a finitely generated subgroup of  a $\Q$-vector space, and therefore a lattice. Now, we prove the lemma remains true for a general choice of the matrices $g_i$, more precisely, we show that making a different choice for the generators of the eigenspaces amounts to modifying the above discrete subgroup of $\Zc=\R^{d^2}$ by a linear transformation: Consider $V_i \coloneqq (A_i^1 \dots A_i^{2d})$ and $W_i \coloneqq (B_i^1 \dots B_i^{2d})$ be a different choice for the bases of the eigenspaces of $\alpha$ and $\frac{1}{\alpha}$ respectively. Then, if we set 
$$V_i=\sum_{j=1}^dl_{ij} v_j\text{, that is } A_i^s=\sum_{j=1}^dl_{ij}a_j^s$$
and
$$W_i=\sum_{j=1}^dk_{ij} w_j\text{, that is } B_i^s=\sum_{j=1}^dk_{ij}b_j^s$$
we may define 
$$E_{ij}^{s,t} = B_i^sA_j^t-A_j^sB_i^t$$
as before, then 
$$E_{ij}^{s,t} = \left(\sum_{p=1}^dk_{ip}b_p^s\right)\left(\sum_{r=1}^dl_{jr}a_r^t\right) - \left(\sum_{r=1}^dl_{jr}a_r^s\right)\left(\sum_{p=1}^dk_{ip}b_p^t\right)=$$
$$=\sum_{1\leq p , r \leq d}k_{ip}l_{jr}(b_p^sa_r^t-a_r^sb_p^t)=\sum_{1\leq p, r \leq d}k_{ip}l_{jr}e_{pr}^{s,t}$$
So the matrix $(E_{ij}^{s,t})$ can be expressed in terms of the matrix $(e_{ij}^{s,t})$ as
$$(E_{ij}^{s,t})=K(e_{ij}^{s,t})L^\top$$
where $K=(k_{ij})$ and $L=(l_{ij})$. 
\end{proof}

\begin{corollary}
The intersection of $\widetilde{\Lambda_\Nc} \coloneqq \langle g_1 \dots g_{2d} \rangle$ with $\Zc$ is a discrete cocompact subgroup of $\Zc$.
\end{corollary}

\begin{proof}
In order to prove discreteness, by the previous lemma, it is enough to show that every element in $\widetilde{\Lambda_\Nc} \cap \Zc$ is contained in the subgroup generated by the commutators $\{g_ig_jg_i^{-1}g_j^{-1} \vert i,j \in \{1 \dots 2d\}\}$. The group $\Lambda_\Vc = \langle \widehat{g_1} \dots \widehat{g_{2d}} \rangle$ (where $\;\widehat{}\;$ denotes the projection on $\Vc$) is the abelianization of $\widetilde{\Lambda_\Nc}$, thus $\widetilde{\Lambda_\Nc} \cap \Zc = [\widetilde{\Lambda_\Nc},\widetilde{\Lambda_\Nc}]$. We have to prove that any element in $[\widetilde{\Lambda_\Nc},\widetilde{\Lambda_\Nc}]$ can be written as a product of commutators $[g_i,g_j]$: Let $h,g,g' \in \widetilde{\Lambda_\Nc}$ , then $[h,gg']=h(gg')h^{-1}(gg')^{-1}=hgg'h^{-1}g'^{-1}g^{-1}=hg(h^{-1}h)g'h^{-1}g'^{-1}g^{-1}=hgh^{-1}(hg'h^{-1}g'^{-1})g^{-1}=hgh^{-1}[h,g']g^{-1}=hgh^{-1}g^{-1}[h,g']=[h,g][h,g']$. The second to last equality is due to the fact that all commutators are central. Using the above relation we may decompose any commutator as desired.\par

We may identify $\Vc$ with $\R^{2d}$ by choosing a basis $e_1 \dots e_d, f_1 \dots f_d$ with respect to which $\widehat{g_s} = (a_1^s \dots a_d^s, b_1^s \dots b_d^s)$, also $\Zc$ may be identified with the subspace $\langle e_i \wedge f_j \vert i,j \in\{1 \dots d\} \rangle$ of $\bigwedge^2 \Vc$. Then $[\widehat{g_s},\widehat{g_t}]_\Vc$ becomes the image of $\widehat{g_s}\wedge\widehat{g_t}$ under the projection on $\Zc$. Since $\{\widehat{g_1} \dots \widehat{g_{2d}}\}$ is a basis in of $\Vc$, then $\{\widehat{g_s} \wedge \widehat{g_t} \vert 1\leq s < t \leq 2d\}$ is a basis of $\bigwedge^2\Vc$ so it generates a cocompact lattice in $\bigwedge^2\Vc$. Therefore, $[\widetilde\Lambda_\Nc,\widetilde\Lambda_\Nc]=\widetilde\Lambda_\Nc \cap \Zc$ is a cocompact subgroup of $\Zc$.
\end{proof}

Finally, we choose the values for $(c_{ij}^k)$ in $g_k$ such that 
$$g_0g_kg_0^{-1}=g_1^{n_{k,1}} \dots g_{2d}^{n_{k,2d}}h_1^{p_{k,1}} \dots h_{d^2}^{p_{k,d^2}}$$
is satisfied, that is for fixed $i$ and $j$, $(c_{ij}^1 \dots c_{ij}^{2d})$ are the solutions to the following system of linear equations
\begin{equation}\label{sys}
\begin{pmatrix}
c_{ij}^1\\
 \vdots\\
 c_{ij}^{2d}
\end{pmatrix}=N
\begin{pmatrix}
c_{ij}^1\\
 \vdots\\
 c_{ij}^{2d}
\end{pmatrix}+
\begin{pmatrix}
f_{ij}^1\\
\vdots\\
f_{ij}^{2d}
\end{pmatrix}+P
\begin{pmatrix}
d_{ij}^1\\
 \vdots\\
 d_{ij}^{d^2}
\end{pmatrix}
\end{equation}
$P$ being the matrix $(p_{i,j})$ and 
$$f_{ij}^k=\sum_{s=1}^{2d}\frac{n_{k,s}(n_{k,s}-1)}{2}b_i^sa_j^s+\sum_{1\leq s<t\leq 2d}n_{k,s}n_{k,t}b_i^sa_j^t$$
Since $1$ is not an eigenvalue of $N$, this has a unique solution.

\hfill

To conclude, to construct a lattice $\Lambda$ of $G$, we proceed as follows:

\begin{itemize}
\item pick a matrix $N\in SL_{2d}(\Z)$ having eigenvalues $\alpha$ and $1/\alpha$ (with $\alpha$ a quadratic positive algebraic integer) both of multiplicities $d$;

\item   pick basis $\langle (a_i^1 \dots a_i^{2d})\vert \ i\in\{1 \dots d\}\rangle$ and  respectively $\langle (b_i^1 \dots b_i^{2d})\vert \ i\in\{1 \dots d\}\rangle$ for the the eigenspaces (over $\R$) associated to the eigenvalues $\alpha$ and $\frac{1}{\alpha}$ respectively.

\item Letting $e_{ij}^{s,t}$ be defined by (\ref{com}) and let
$$z^{s,t}:=
\begin{pmatrix}
I_d & \rvline & 0 & \rvline & (e_{ij}^{s,t})\\
\hline
0 & \rvline & 1 & \rvline & 0 \\
\hline
0 & \rvline & 0 & \rvline& I_d
\end{pmatrix}
.$$ 
We see that $\lambda_\Zc\coloneqq \langle\{z^{s,t}\}_{s, t=1, 2d}\rangle$ is discrete lattice in $\Zc$. Consider an arbitrary cocompact lattice $D\coloneqq \{h_k\}_{k=1 \dots d^2}$
$$h_k=
\begin{pmatrix}
I_d & \rvline & 0 & \rvline & (d_{ij}^k)\\
\hline
0 & \rvline & 1 & \rvline & 0 \\
\hline
0 & \rvline & 0 & \rvline& I_d
\end{pmatrix}\text{, }k\in \{1 \dots d^2\}
$$
such that $D\supset \{z^{s,t}\}_{s,t \in \{1\dots 2d\}}$.

\item consider an arbitrary matrix $P\in {\rm Mat}_{2d, d^2}(\Z)$ and let  $(c_{ij}^1 \dots c_{ij}^{2d})$  be the solutions of the system (\ref{sys}).
Define the matrices $$
 g_k=
\begin{pmatrix}
I_d & \rvline & (b_j^k) & \rvline & (c_{ij}^k)\\
\hline
0 & \rvline & 1 & \rvline & (a_i^k) \\
\hline
0 & \rvline & 0 & \rvline& I_d
\end{pmatrix}\text{, }k\in \{1 \dots 2d\}$$ and let $g_0$ be defined as in (\ref{g0}).
Then the subgroup  $$\Lambda\coloneqq \langle g_0, \{g_k\}_{k=1, \dots, 2d}, \{h_k\}_{k=1,\dots d^2}\rangle$$ is a cocompact lattice in $G$.

\end{itemize}

\begin{definition}
Let $d$ be an odd integer, set $n=\frac{(d+1)^2}{2}$ and let $\Lambda$ be a discrete cocompact subgroup of $G$ as above. We will denote by $X_{d,\Lambda}$ the compact complex manifold obtained as the quotient of $G=\HH \times \C^{n-1}$ by the left action of $\Lambda$.  
\end{definition}

\begin{remark}
When $d=1$, the manifold $X_{d,\Lambda}$ is an Inoue surface of type $S^{(+)}_{N,p,q,r;0}$. In this case, the cocompact lattice $D$ containing $\lambda_\Zc$ must be of the form $D=\frac{1}{r}\lambda_\Zc.$ The matrix $P$ is just 
$P=\left(\begin{array}{ccc}
p\\
q
\end{array}
\right)$;
see Inoue's original paper \cite{Ino}.
\end{remark}

\section{Betti numbers}
The manifold $X_{d,\Lambda}$ is diffeomorphic to a fiber bundle over a circle ($X_{d,\Lambda} \to \R_{>}/\langle \alpha \rangle$) and the fiber is $\Mc\coloneqq \Nc/\Lambda_\Nc$. In turn, $\Mc$ is a fiber bundle over a $2d$-torus 
$$\R^{2d}/\langle (a_1^k \dots a_d^k, b_1^k \dots b_d^k) \vert k\in\{1 \dots 2d\}\rangle$$ with fiber a $d^2$-torus ($\Zc/D$). We will compute the first two Betti numbers of $X_{d,\Lambda}$ using group cohomology and the spectral sequences associated to these two fiber bundles. The key tool is the following classical theorem (see \cite{Mum}):

\begin{theorem}\label{mumt}
\label{Mum}
Let $X$ be a topological space and $G$ a discrete group, acting freely and discontinously on $X$. Let $Y=X/G$ and $\pi:X \to Y$ be the projection, then for any sheaf $\F$ on $Y$, there is a natural homomorphism
$$\phi:H^p(G,\Gamma(X,\pi^*\F)) \to H^p(Y,\F)$$
If $H^i(X, \pi^*\F)=0$, $i \geq 1$, then the map $\phi$ is an isomorphism. Furthermore, the maps $\phi$ are compatible with cup products.
\end{theorem}

We start by studying the cohomology of the fiber $\Mc\simeq (\R^{2d}\times T^{d^2})/F$, where $F(\simeq \Z^{2d})\subset \Vc$ is the group generated by $\{\widehat{g_1}\dots \widehat{g_{2d}}\}$. Consider the spectral sequence starting from the page zero given by $E_0^{p,q}\simeq C^p(F, \Gamma(\mathcal{A}^q(\R^{2d}\times T^{d^2})))$ (where ${\mathcal A}^q$ is the sheaf of differential $q-$forms), $d$ denoting the vertical differential and $\delta$ the horizontal one. Here, $E_0^{p,q}$ is the set of all functions $F^{p} \to \Gamma(\mathcal{A}^q(\R^{2d}\times T^{d^2}))$, $\delta$ is the differential of the bar resolution and $d$ is the usual exterior differential of forms. By the previous theorem, taking cohomology along the lines of the bicomplex will evaluate the cohomology of the sheaves $\mathcal{A}^q$ on $\Mc$. Since these sheaves are fine, we know that the spectral sequence will converge to the cohomology of $\Mc$. The second page of the spectral sequence is $E_2^{p,q}\simeq H^p(F,H^q(T^{d^2},\C))$ and $F$ acts trivially on the cohomology of $T^{d^2}$, therefore $E_2^{p,q}\simeq H^p(F,\C)\otimes H^q(T^{d^2},\C)$.

\begin{remark}
If a cohomological spectral sequence with multiplicative structure has the second page multiplicatively generated by elements of total degree one, then $E_3^{p,q}\simeq E_{\infty}^{p,q}$. Moreover, in order to compute page $3$, one only needs to know the map $d_2:E_2^{0,1} \to E_2^{2,0}$. This applies to our spectral sequence and it is true in general for spectral sequences coming from principal fiber bundles with a torus fiber over a torus base. 
\end{remark}

The generators of $H^0(F,H^1(T^{d^2},\C))$ are represented by $\{\phi_{ij}:id_{F} \to dc_{ij}\}$. Let $(k_1\dots k_{2d})\in \Z^{2d} \simeq F$, then:
$$\delta(\phi_{ij})(k_1\dots k_{2d})=(k_1\dots k_{2d})\cdot \phi_{ij}-\phi_{ij}=((g_1^{k_1}\dots g_{2d}^{k_{2d}})^{-1})^{*}dc_{ij}-dc_{ij}=\left(-\sum_{s=1}^{2d}k_sb_j^s\right)da_i$$
Let $\psi_{ij}$ be such that $d(\psi_{ij})=\delta(\phi_{ij})$, for example:
$$\psi_{ij}(k_1\dots k_{2d})=\left(-\sum_{s=1}^{2d}k_sb_j^s\right)a_i$$
Let $(l_1 \dots l_{2d})\in \Z^{2d} \simeq F$, we compute $\delta(\psi_{ij})$:
$$\delta(\psi_{ij})((k_1\dots k_{2d}),(l_1 \dots l_{2d}))=$$ $$=((g_1^{k_1}\dots g_{2d}^{k_{2d}})^{-1})^{*}\psi_{ij}(l_1\dots l_{2d})-\psi_{ij}(k_1+l_1 \dots k_{2d}+l_{2d})+\psi_{ij}(k_1\dots k_{2d})=$$
$$=\left(-\sum_{s=1}^{2d}l_sb_j^s\right)\left(a_i-\sum_{t=1}^{2d}k_ta_i^t\right)-\left(-\sum_{s=1}^{2d}(k_s+l_s)b_j^s\right)a_i+\left(-\sum_{s=1}^{2d}k_sb_j^s\right)a_i=$$
$$=\left(-\sum_{s=1}^{2d}l_sb_j^s\right)\left(-\sum_{t=1}^{2d}k_ta_i^t\right)$$

For $\delta(\psi_{ij})$ determined above, we have $d_2([\phi_{ij}]_2)=[\delta \psi_{ij}]_2$. By identifying $E_2^{2,0}$ with $H^2_{DR}(T^{2d},\C)$, $d_2([\phi_{ij}]_2)$ becomes 
$$\sum_{1\leq s < t \leq 2d}(\delta(\psi_{ij})(\hat{g_s},\hat{g_t})-\delta(\psi_{ij})(\hat{g_t},\hat{g_s}))du_s\wedge du_t=\sum_{1\leq s < t \leq 2d}(b_j^sa_i^t-b_j^ta_i^s)du_s\wedge du_t$$
where $u_i$ are the coordinate functions on $\R^{2d}$ with respect to the basis $\{\widehat{g_i}\}$. Rewriting this in terms of $a_i$ and $b_i$, we get 
$$\sum_{1\leq s < t \leq 2d}(b_j^sa_i^t-b_j^ta_i^s)du_s\wedge du_t = db_j \wedge da_i$$
This proves that $E_3^{0,1}=0$ and $E_3^{2,0}=(\bigwedge^2\langle da_1\dots da_d \rangle)\oplus (\bigwedge^2\langle db_1 \dots db_d\rangle )$, so $H^1(\Mc,\C)=\langle da_i, db_j \rangle$ and $H^2(\Mc,\C)=\bigwedge^2\langle da_i\rangle \oplus \bigwedge^2\langle db_j \rangle \oplus \langle da_i \otimes dc_{ij}\rangle \oplus \langle db_j \otimes dc_{ij} \rangle$.\par
We may now compute the first two Betti numbers of $X_{d,\Lambda}$. Consider the following change in coordinates on $G$: $b_i^{'}=\frac{b_i}{\alpha}$, for $i\in\{1\dots d\}$ and all the other coordinates remain unchanged. This presents  the manifold $G/\Lambda_\Nc$ as the product $\R_>\times \Mc$, where $\Mc$ is as above. This time, we are interested in the spectral sequence whose second page is $E_2^{p,q}=H^p(\Z,H^q(\Mc,\C))$, where the generator of $\Z$ acts as $g_0$. Since $H^1(\Mc,\C)$ is finite-dimensional, the dimensions of $H^0(\Z,H^1(\Mc,\C))$ and $H^1(\Z,H^1(\Mc,\C))$ are equal - since they are the kernel and cokernel of $(g_0^{-1})^*- id$ as an endomorphism of $H^1(\Mc,\C)$. No elements of $H^1(\Mc,\C)$ are invariant under the action of $g_0$ (since $(g_0^{-1})^*(da_i)=\frac{1}{\alpha} da_i$ and $(g_0^{-1})^*(db_i^{'})=\alpha db_i^{'}$), thus $H^0(\Z,H^1(\Mc,\C))=H^1(\Z,H^1(\Mc,\C))=0$. Similarly, $H^0(\Z,H^2(\Mc,\C))=0$ and $H^1(\Z,H^0(\Mc,\C))=H^1(\Z,\C)=\C$. We have proved the following:

\begin{proposition}
The first two Betti numbers of $X_{d,\Lambda}$ are $b_1=1$ and $b_2=0$.
\end{proposition}

\begin{corollary}
The manifolds $X_{d,\Lambda}$ are non-K\"ahler.
\end{corollary}

\section{Analytic properties}
In this section we will study a certain class of manifolds of type $X_{d,\Lambda}$. Let $\Lambda$ be a discrete subgroup cocompact of $G$ as in the first section, consider the following properties of such a subgroup:

\begin{definition}
1) We say that $\Lambda$ is of toroidal type if the quotient of $\C^{n-m}$ by the subgroup $D$ is a toroidal group. Here, $\C^{n-m}$ is the subspace of $G$ given by the variables $z_{ij}$ with $i\neq 0$ and $m=\frac{d-1}{2}$, as before.\\
2) We say that $\Lambda$ is of algebraic type if all the entries in the elements of $\Lambda$ are algebraic numbers.
\end{definition}

\begin{remark}
For $d=1$, the subgroup $\Lambda$ cannot be of toroidal type.
\end{remark}

\textbf{Examples:}
\begin{itemize}
\item[i)] Let $d=3$, and let $\Lambda$ be chosen as in the special case of Lemma \ref{dis}. That is, if we represent the matrices in $D$ by their upper right $3\times3$ block, then $D$ is of finite rank in $q\langle T_{ij}-T_{ji}, \sqrt{\delta}T_{ij}+\sqrt{\delta}T_{ji}\vert 1\leq i,j \leq 3\rangle$, where $q$ is a rational, nonzero  number. To check if $\Lambda$ is of toroidal type, we only need to check if the first columns of these matrices generate a dense subgroup in $\R^3$ (see \cite{Abe}). The matrix formed with these columns is:
$$\begin{pmatrix}
1 & 0 & 0 & \sqrt{\delta} & 0 & 0 & 0 & 0 & 0\\
0 & 1 & 0 & 0 & \sqrt{\delta} & 0 & 0 & 0 & 0\\
0 & 0 & 0 & 0 & 0 & 0 &  \sqrt{\delta} & 0 & 0\\
\end{pmatrix}$$
They do not generate a dense subgroup, so, in this case $\Lambda$ is not of toroidal type.
\item[ii)] Let $d$ and $N\in SL_6(\Z)$ be the same as for the previous example, but this time change the bases of the eigenspaces by the matrices $L$ and $K$ (again, see \ref{dis}). This time, we shall look at the subgroup generated by the first columns of the matrices $K(T_{ij}-T_{ji})L^\top, K(\sqrt{\delta}T_{ij}+\sqrt{\delta}T_{ji})L^\top$. In fact, multiplying with $K$ from the right will not affect the property of this subgroup of being dense, since it changes $\R^3$ by a linear transformation, so we may consider $K=I_3$. As above, we now have:
$$\begin{pmatrix}
l_{11} & 0 & l_{12} & \sqrt{\delta}l_{11} &0 & \sqrt{\delta}l_{12} & 0 & 0 & \sqrt{\delta}l_{13}\\
0 & l_{11} & -l_{13} & 0& \sqrt{\delta}l_{11} & \sqrt{\delta}l_{13} & 0& \sqrt{\delta}l_{12}&0\\
-l_{13} & -l_{12} & 0 & \sqrt{\delta}l_{13}& \sqrt{\delta}l_{12} & 0&\sqrt{\delta}l_{11}&0&0\\
\end{pmatrix}$$ 
which is plainly dense for generic choices of $L$ and $K.$  Notice that if $L$ and $K$ have algebraic entries then $\Lambda$ is   of algebraic type.
\end{itemize}

If $\Gamma$ is a discrete subgroup of $\C^n$ such that $\C^n/\Gamma$ is toroidal and all the entries of the period matrix representing $\Gamma$ are algebraic numbers, then $\C^n/\Gamma$ satisfies Hodge decomposition. Moreover, this is also true for every domain $U$ of $\C^n/\Gamma$ such that its inverse image $\tilde{U}$ in $\C^n$ is convex (see \cite{OtTo}). This result will be useful in proving the following:

\begin{proposition} \label{iso}
Let $X$ be a manifold of the form $G/\Lambda$, such that $\Lambda$ is of toroidal and algebraic type, then the map $\zeta:H^1(X,\C) \to H^1(X,\OO_X)$, induced by the sheaf inclusion $\C \to \OO_X$, is an isomorphism.
\end{proposition}

\begin{proof}
In order to study $H^1(X, \OO_X)$ we observe that it is, by Mumford's theorem \ref{mumt}, isomorphic to $H^1(\Lambda, A)$, where $A=\Gamma(\OO_X)$ (a left $\Gamma$-module with action defined as $ga=(g^{-1})^*a$, for $g\in\Lambda$ and $a\in A$).
Using the short exact sequence
\[\begin{tikzcd}
	1 & {\Lambda_\Nc} & \Lambda & {\langle g_0 \rangle} & 1
	\arrow[from=1-1, to=1-2]
	\arrow[from=1-2, to=1-3]
	\arrow[from=1-3, to=1-4]
	\arrow[from=1-4, to=1-5]
\end{tikzcd}\]
we may apply the Lyndon-Hochschild-Serre spectral sequence (LHS for short) and get 
\[\begin{tikzcd}
	0 & {H^1(\langle g_0 \rangle, A^{\Lambda_\Nc})} & {H^1(\Lambda,A)} & {H^1(\Lambda_\Nc,A)^{\langle g_0 \rangle}} & {H^2(\langle g_0 \rangle, A^{\Lambda_\Nc})}
	\arrow[from=1-1, to=1-2]
	\arrow[from=1-2, to=1-3]
	\arrow[from=1-3, to=1-4]
	\arrow[from=1-4, to=1-5]
\end{tikzcd}\]
First, we address the first and fourth term of this sequence. The elements of $A^{\Lambda_\Nc}$ are holomorphic functions on the quotient $G/\Lambda_\Nc$; let $f$ be such a function. For fixed $w_2$, $f$ is bounded since the (real) fibers of $G/\Lambda_\Nc\to \Vc/\Lambda_{\Vc}$ are compact, this implies that $f$ does not depend on the $z_{ij}$ coordinates. Moreover, the subgroup generated by $\{a_1^i \vert i\in\{1 \dots 2d\}\}$ in $\R$ is dense so $f$, being continous, is constant with respect to $w_1$, the real part of $w$, therefore $f$ is constant.\par
We have shown that $A^{\Lambda_\Nc}=\C$ (and henceforth the action of $\langle g_0 \rangle$ on it is trivial). Then 
$$H^1(\langle g_0 \rangle, A^{\Lambda_\Nc})=H^1(\Z,\C)=\C \ \ \text{ and } \ \ H^2(\langle g_0 \rangle, A^{\Lambda_\Nc})=H^2(\Z,\C)=0.$$
 We see henceforth that $\zeta$ is injective, for example, by using the four lemma on the following diagram
\[\begin{tikzcd}
	0 & {\mathbb{C}} & {H^1(\Lambda, \mathbb{C})} & 0 \\
	0 & {H^1(\langle g_0 \rangle, A^{\Lambda_\Nc})} & {H^1(\Lambda, A)} & {H^1(\Lambda_\Nc, A)^{\langle g_0 \rangle}}
	\arrow[from=1-1, to=1-2]
	\arrow[from=1-2, to=1-3]
	\arrow[from=1-2, to=2-2, equal]
	\arrow[from=1-3, to=1-4]
	\arrow["\zeta" , from=1-3, to=2-3]
	\arrow[from=1-4, to=2-4]
	\arrow[from=2-1, to=2-2]
	\arrow[from=2-2, to=2-3]
	\arrow[from=2-3, to=2-4]
\end{tikzcd}\]

The surjectivity of $\zeta$ is equivalent to $H^1(\Lambda_\Nc,A)^{\langle g_0 \rangle}=0$. In the following lemma we will show that the map $H^1(\Lambda_\Nc, \C) \to H^1(\Lambda_\Nc, A)$ is surjective. As  $H^1(\Lambda_\Nc, \C)$ is finite-dimensional it follows that $H^1(\Lambda_\Nc,A)$ is finite-dimensional too. Since  the action of $g_0$ on $H^1(\Lambda_\Nc, \C)$  is diagonalisable and has not $1$ as an eigenvalue (see the arguments in the previous section), it follows that the action of $g_0$ on $H^1(\Lambda_\Nc, A)$ also has no fixed points, henceforth $H^1(\Lambda_\Nc, A)^{\langle g_0\rangle }=0.$
\end{proof}

Hence it remains to prove the
\begin{lemma}\label{lemuta}
The following maps are surjective:

\begin{itemize}
\item[i)] $\beta:H^1(F,\C) \to H^1(F, A^D)$
\item[ii)] $\gamma:H^1(D,\C) \to (H^1(D,A))^F$
\item[iii)] $\mu:H^1(\Lambda_\Nc,\C) \to H^1(\Lambda_\Nc, A)$
\end{itemize}

\end{lemma}

\begin{proof}

i)
To begin, recall that $ F$ is generated by $\widehat{g_1} \dots \widehat{g_{2d}}$ and that the action of these on $\HH\times \C^m$ (where  $\C^m$ is the space generated by $z_{01} \dots z_{0m}$) is done by translations. We can rewrite $F$ as the group generated by some  $e_1 \dots e_d$ and respectively $e'_{1} \dots e'_{d}$ where the $e_1 \dots e_d$ are chosen in such a way that
$\HH\times \C^m/ \langle e_1 \dots e_d \rangle $ is a (convex open subset of a) toroidal group. \par 

The toroidality is equivalent to the condition that the action of $\langle e_1 \dots e_d \rangle$ on the real line $Re(w)$ has dense orbits, we see that such a choice is possible because the group generated by $a_1^1 \dots a_1^{2d}$ is dense.

Moreover, since $\Lambda$ was chosen of algebraic type, we see that the entries of the vectors $e_i$ are algebraic numbers (with respect to the canonical basis of $\C\times \C^m$). Then the elements of 
$$H^1(\langle e_1 \dots e_d\rangle, A^D)$$
are images of elements in 
$H^1(\langle e_1 \dots e_d\rangle, \C)$ (cf \cite{OtTo}, Thm. 3.1.).

Let  $\omega \in H^1(F, A^D):$ then its image via the restriction map
$$res: H^1(F, A^D)\to H^1(\langle e_1 \dots e_d\rangle, A^D)$$
 takes constant values, henceforth $\omega$ takes constant values on elements in $\langle e_1,\dots, e_d\rangle$. Hence,  we are left to show that $\omega$ takes also constant values the remaining $e'_1, \dots, e'_d$. But $\omega$ is a cocycle, hence
 $$\omega(e_i+ e'_j)=\omega(e_i)+e_i\omega(e'_j)$$
 and also
 $$\omega(e_i+ e'_j)=\omega(e'_j)+e'_j\omega(e_i)$$
 hence
$$ \omega(e_i)+e_i\omega(e'_j)=\omega(e'_j)+e'_j\omega(e_i)$$
As $\omega(e_i)$ is constant, $e'_j\omega(e_i)=\omega(e_i)$ hence we get
$$ e_i\omega(e'_j)=\omega(e'_j)$$
hence $\omega(e'_j)$ is invariant w.r. to the $e_1 \dots e_d$ hence from the toroidality assumption we get that
$\omega(e'_j)$ is also constant.

\hfill

ii)
$D$ acts by translations on $\C^{d(m+1)}$  where $\C^{d(m+1)}$ is the space of the variables $z_{ij}$ with $i\not=0.$
Since, by our assumptions all the entries of the lattice $D$ are algebraic, $\left(\C^{d(m+1)}/D\right)$ is a toroidal group {\em of finite type}, that is 
all the cohomology groups $$H^i\left(\left(\C^{d(m+1)}/D\right), \OO\right), i\geq 1$$ are finite-dimensional.
Next, notice that 
$\HH\times \C^n/D$ can be written as $(\HH\times \C^m)\times \left(\C^{d(m+1)}/D\right)$
that is, it is  a product between  a domain in $\C^{m+1}$ and a toroidal group.

By Mumford's theorem \ref{mumt}, $H^1(D,A)\simeq H^1(\HH\times \C^{n}/D, \OO_{\HH\times \C^{n}/D})$ and  as $\C^{d(m+1)}/D$ is toroidal of finite type the dimension of $H^i(\C^{d(m+1)}/D, \OO_{\C^{d(m+1)}/D})$ is finite for $i \geq 1$. As before, we know this because of the assumption that the entries al algebraic. To determine $H^1(D,A)$, we will use the following result by Kazama and Umeno (\cite{Kaz}):

\begin{theorem}
Let $T$ be a toroidal group of finite type and $U$ a polydisk in $\C^n$ ($U=\left\{(z_1\dots z_n) \vert |z_i| \leq d_i, i\in\left\{1 \dots n\right\} \right\}$, for some polyradius $(d_1 \dots d_n)$), then
$$H^1(T \times U, \OO_{T \times U}) \simeq H^1(T, \OO_T) \otimes H^0(U, \OO_U)$$
\end{theorem}
In fact, it is easy to see that the above theorem applies also when  the some of the radii $d_i$ are infinite.
Applying the theorem to our case, we get
$$H^1(\HH\times \C^{n}/D, \OO_{\HH\times \C^{n}/D})\simeq H^1(\C^{d(m+1)}/D, \OO_{\C^{d(m+1)}/D})\otimes H^0(\HH\times \C^m, \OO_{\HH\times \C^m}).$$
By taking $F-$invariants, we get that  $H^0(\HH\times \C^m, \OO_{\HH\times \C^m})^F=\C$, 
hence
$$H^1(\C^{d(m+1)}/D, \OO_{\C^{d(m+1)}/D})\simeq (H^1(D, A))^F$$
and the conclusion follows since we know, by the assumption of toroidality of finite type, that we have a surjection
$$H^1(D,\C)\to H^1(\C^{d(m+1)}/D, \OO_{\C^{d(m+1)}/D}).$$

iii) Consider the restriction map $res: H^1(\Lambda_\Nc,A) \to H^1(D,A)^{F}$; by the surjectivity of $\gamma$, every element of $H^1(D,A)^{F}$ is represented by a cocycle of constants. Thus, every element of $H^1(\Lambda_\Nc,A)$ is represented by a cocycle taking constant values on $D$.\par

Let $\omega$ be such a cocycle and let $g \in \Lambda_\Nc$ and $h \in D$, since $h$ is central, we have 
$$\omega(gh)=\omega(hg) \implies g\omega(h)+\omega(g)=h\omega(g)+\omega(h)$$ 
$\omega(h)$ is constant and thus $\omega(g)=h\omega(g)$. So $\omega(g) \in A^D$ (that is, a holomorphic function which depends only on the variables $w, z_{01} \dots z_{0m}$). We have shown that the map $H^1(\Lambda_\Nc, A^D) \to H^1(\Lambda_\Nc, A)$ is surjective.\par
We next study the space $H^1(\Lambda_\Nc, A^D)$. To do this,  consider the diagram:
\begin{equation}\label{diag}
\begin{tikzcd}
	{H^1(F,\mathbb{C})} & {H^1(\Lambda_\Nc,\mathbb{C})} & {H^1(D,\mathbb{C})} \\
	{H^1(F, A^D)} & {H^1(\Lambda_\Nc,A^D)} & {H^1(D,A^D)^{F}} \\
	{H^1(F,\Gamma(d\mathcal{O}_{\mathbb{H}\times\mathbb{C}^m}))} & {H^1(\Lambda_\Nc,\Gamma(d\mathcal{O}_{\mathbb{H}\times\mathbb{C}^m}))} & {H^1(D,\Gamma(d\mathcal{O}_{\mathbb{H}\times\mathbb{C}^m}))^{F}}
	\arrow[from=1-1, to=1-2]
	\arrow["\beta" , from=1-1, to=2-1]
	\arrow[from=1-2, to=1-3]
	\arrow[from=1-2, to=2-2]
	\arrow["\chi", from=1-3, to=2-3]
	\arrow["inf" , from=2-1, to=2-2]
	\arrow["d" , from=2-1, to=3-1]
	\arrow["res" , from=2-2, to=2-3]
	\arrow["d" , from=2-2, to=3-2]
	\arrow["d" , from=2-3, to=3-3]
	\arrow["inf" , from=3-1, to=3-2]
	\arrow["res", from=3-2, to=3-3]
\end{tikzcd}
\end{equation}
First of all, some remarks about the diagram:

\begin{itemize}
\item[a)] The map $\beta$ is surjective, as proven in point (i) of the lemma.

\item[b)] The map $\chi$ is an isomorphism. This is true because $A^D$ is a trivial $D$-module and $D$ is contained in the center of $\Lambda_\Nc$, so, by the definition of the action of $F$ on $H^1(D,A^D)$, the term $H^1(D,A^D)^F$ contains precisely the cocycles taking constant values (as they are the only $F$-invariant functions).

\item[c)] The rows are exact, each being an inflation-restriction sequence (or low degree terms in the LHS spectral sequence) coming from  
\[\begin{tikzcd}
	1 & F & {\Lambda_\Nc} & D & 1
	\arrow[from=1-1, to=1-2]
	\arrow[from=1-2, to=1-3]
	\arrow[from=1-3, to=1-4]
	\arrow[from=1-4, to=1-5]
\end{tikzcd}\]

\item[d)] The columns are exact, for the first two this is clear from the long exact sequence in cohomology induced by 
\[\begin{tikzcd}
	0 & {\mathbb{C}} & {A^D} & {\Gamma(d\mathcal{O}_{\mathbb{H}\times\mathbb{C}^m})} & 0
	\arrow[from=1-1, to=1-2]
	\arrow[from=1-2, to=1-3]
	\arrow[from=1-3, to=1-4]
	\arrow[from=1-4, to=1-5]
\end{tikzcd}\]
seen as a short exact sequence of $F$-modules, or $\Lambda_\Nc$-modules respectively. For the last column, we start with the same short exact sequence, this time treating them as trivial $D$-modules and, after taking cohomology, we take $F$-invariants. This implies that the last column is a complex and since $\chi$ is an isomorphism, we get that it is in fact exact.  
\end{itemize}
Let $\nu$ be a cocycle representing a class in $H^1(\Gamma_H, A^D)$, then, since the bottom-right map labeled $d$ in (\ref{diag})  is zero (this is because $\chi$ is surjective) we have:
$$res(d(\nu)) = d(res(\nu)) = 0$$
Thus, $d(\nu)$ comes from the left (it is contained in the image of the inflation map). By the forthcoming lemma (\ref{lemuca}), every element of $H^1(F, \Gamma(d\OO_{\HH \times \C^m}))$ is represented by a cocycle, taking values in $\langle dw, dz_{01} \dots dz_{0m} \rangle$, so $d(\nu)$ is cohomologous to a cocycle taking values in $\langle dw, dz_{01} \dots dz_{0m} \rangle$. From this, we conclude that, up to cohomology, $\nu$ may be chosen to take values in $\Gamma(\text{Aff}_{\HH \times \C^m})$. We have shown that the maps
$$H^1(\Lambda_\Nc, \Gamma(\text{Aff}_{\HH \times \C^m})) \to H^1(\Lambda_\Nc, A^D) \to H^1(\Lambda_\Nc, A) $$
are both surjective.\par
To finish the proof of point iii) in lemma  \ref{lemuta}, we will study the following diagram:

\begin{equation}\label{diagmic}
{\tiny
\begin{tikzcd}
	{H^1(\Lambda_\Nc,\C)} & {H^1(\Lambda_\Nc, \Gamma (\text{Aff}_{\HH \times \C^m}))} & {H^1(\Lambda_\Nc,\Gamma( d\text{Aff}_{\HH \times \C^m}))} & {H^2(\Lambda_\Nc,\C)} \\
	{H^1(\Lambda_\Nc, \C)} & {H^1(\Lambda_\Nc,A)} & {H^1(\Lambda_\Nc,dA)} & {H^2(\Lambda_\Nc,\C)}
	\arrow[from=1-1, to=1-2]
	\arrow[from=1-1, to=2-1, equal]
	\arrow[from=1-2, to=1-3]
	\arrow[two heads, from=1-2, to=2-2]
	\arrow["\eta", from=1-3, to=1-4]
	\arrow["\epsilon", from=1-3, to=2-3]
	\arrow[from=1-4, to=2-4, equal]
	\arrow[from=2-1, to=2-2]
	\arrow[from=2-2, to=2-3]
	\arrow[from=2-3, to=2-4]
\end{tikzcd}
}
\end{equation}

We want to show that the map $H^1(\Lambda_\Nc, A) \to H^1(\Lambda_\Nc, dA)$ is zero. Since we know that the second vertical map in (\ref{diagmic}) is surjective it is enough to prove that for any class $[\omega] \in H^1(\Lambda_\Nc, \Gamma(d \text{Aff}_{\HH \times \C^m}))$ such that $\eta([\omega]) = 0$, we have $\epsilon([\omega])=0$. Let the representative $\omega$ be defined as follows:
\begin{equation}\label{omega}
\omega(g_t) = S_0^t dw + S_1^t dz_{01} + \dots + S_m^t dz_{0m}\text{, for all }t \in \{1\dots 2d\}
\end{equation}
This implies that 
\begin{equation}\label{state}
\omega(h_s) = 0\text{, for all } s \in \{1 \dots d^2\}
\end{equation}
The cochain $\omega$ is closed and the action of $\Lambda_\Nc$ on $\Gamma(d\text{Aff}_{\HH \times \C})$ is trivial, thus $\omega(g_sg_tg_s^{-1}g_t^{-1})=0$, for all $s,t \in \{1 \dots 2d\}$. As shown in the first section, for each element $h_s$ there exists a positive integer $r$ such that $h_s^r$ is contained in the subgroup generated by the commutators of the $g_t$'s. Thus, the above statement (\ref{state}) is true.

Also, denote by $\omega_i$ the $i$'th component of $\omega$ in the decomposition given by (\ref{omega}),  $\omega_i(g_s)=S_i^s$ (notice that  $[\omega_i]$  can be viewed as a class in $H^1(\Lambda_\Nc, \C)$).
 We may identify the classes $da_i$, $db'_j \in H^1_{DR}(G/\Lambda_\Nc, \C)$ with the group-cohomological classes represented by $\varphi_i$ and $\psi_j$ respectively, where
$$\varphi_i(g_s) = -a_i^s\text{ and }\psi_j(g_s)=-b_j^s.$$
Then, the image of $[\omega]$ through the connecting map $\eta$ is represented by the cocycle
$$(\eta(\omega))(g_s, g_t)= - (S_0^ta_1^s + S_1^t(a_2^s + \sqrt{-1}a_3^s) + \dots + S_m^t(a_{d-1}^s + \sqrt{-1}a_d^s))=$$
$$= \varphi_1 \smile \omega_0 + (\varphi_2 + \sqrt{-1}\varphi_3) \smile \omega_1 + \dots + (\varphi_{d-1} + \sqrt{-1}\varphi_d)\smile \omega_m$$  
where by $\smile$ we have denoted the cup product. Assuming that $\eta(\omega)$ is exact, we get, by the computations in section 1, that 
$$\omega_i=B_1^i \psi_1 + \dots B_d^i \psi_d + A_0^i \varphi_1 + A_1^i(\varphi_2 + \sqrt{-1}\varphi_3)+ \dots + A_m^i(\varphi_{d-1}+\sqrt{-1}\varphi_d)$$
and $A_j^i = A_i^j$. We need to prove that that $\epsilon(\omega)$ is exact. Let $\mu \in C^0(\Lambda_\Nc, dA)$ be the $0$-cochain given by the form:
$$d\left( \sum_{\substack{0\leq i \leq m\\ 1\leq j \leq d}}B_j^iz_{ji}\right) + d\left(\frac{\displaystyle\sum_{1\leq i, j \leq m}A_j^i z_{0i}z_{0j} + \displaystyle\sum_{1\leq i \leq m}(A_i^0 + A_0^i) w z_{0i} + A_0^0 w^2}{2}\right)$$ 
By a simple computation we can see that $\delta(\mu)=\epsilon(\omega)$, this finishes the proof.
\end{proof}

\begin{lemma}\label{lemuca}
Let $\Lambda$ be an additive subgroup of $\C^n$ (not necessarily discrete) and $U \subseteq \C^n$ an open convex subset, invariant under the action of $\Lambda$. If $\Gamma(\OO_U)^\Lambda = \C$ and the map $H^1(\Lambda, \C) \to H^1(\Lambda, \Gamma(\OO_U))$ is surjective, then 
$$
H^1(\Lambda, \Gamma(d{\rm Aff}_U)) \to H^1(\Lambda, \Gamma(d\OO_U))
$$
is surjective. That is, every class in $H^1(\Lambda, \Gamma(d\OO_U))$ is represented by a cocycle taking values in $\langle dz_1 \dots dz_n \rangle$.
\end{lemma}
 
\begin{proof}
We know that $\Gamma(\Omega^1_U)$ and  $(\Gamma(\OO_U))^n$  are isomorphic as vector spaces, and since $\Lambda$ acts by translations, they are also isomorphic as $\Lambda$-modules. Consider the diagram:
\[\begin{tikzcd}
	& {H^1(\Lambda, \C^n)} \\
	{H^1(\Lambda, \Gamma(d\OO_U))} & {H^1(\Lambda, (\Gamma(\OO_U))^n)} \\
	& 0
	\arrow["\phi" ', from=1-2, to=2-1]
	\arrow[from=1-2, to=2-2]
	\arrow[from=2-1, to=2-2]
	\arrow[from=2-2, to=3-2]
\end{tikzcd}\]
The horizontal map is the one induced by the inclusion $\Gamma(d\OO_U) \to \Gamma(\Omega^1_U) \simeq (\Gamma(\OO_U))^n$. Then $\phi$ is induced by the map $\C^n \to \Gamma(d\OO_U)$, sending $(a_1 \dots a_n)$ to $a_1 dz_1 + \dots + a_ndz_n$. We have to show that $\phi$ is surjective.\par
The vertical map is surjective by hypotheis, hence if $\nu$ is a cocycle representing a class in $H^1(\Lambda, (\Gamma(\OO_U))^n)$ then
$$
\nu(g) = (a_1^g \dots a_n^g) + (g(f_1 \dots f_n) - (f_1 \dots f_n))\text{, for all } g \in \Lambda
$$
Here $g \to (a_1^g \dots a_n^g)$ is cocycle in $H^1(\Lambda, \C^n)$ and $f_1 \dots f_n$ are holomorphic functions not depending on $g$. If $\nu$ comes from $H^1(\Lambda, \Gamma(d\OO_U))$ then 
$$\frac{\partial}{\partial z_s} (gf_t - f_t) = \frac{\partial}{\partial z_t} (gf_s - f_s)\text{, for all }s,t \in \{1 \dots n\} $$
As $g$ is a translation, it commutes with the partial derivatives so we get
$$g\left(\frac{\partial f_t}{\partial z_s}-\frac{\partial f_s}{\partial z_t}\right) = \frac{\partial f_t}{\partial z_s}-\frac{\partial f_s}{\partial z_t}$$
Thus, $\frac{\partial f_t}{\partial z_s}-\frac{\partial f_s}{\partial z_t}$ is a holomorphic function, invariant under the action of $\Lambda$ and therefore constant.\par
The linear system formed by the equations $\frac{\partial f_t}{\partial z_s}-\frac{\partial f_s}{\partial z_t} = c_{st}$ admits a solution $(f_1' \dots f_n')$ with $f_1' \dots f_n'$ linear. And the assciated system of homogenous equations $\frac{\partial f_t}{\partial z_s}-\frac{\partial f_s}{\partial z_t}=0$ has solutions of the form $(\frac{\partial F}{\partial z_1} \dots \frac{\partial F}{\partial z_n})$, where $F$ is a holomorphic function on $U$.\par
From the above computation, we get that
$$\nu(g) = (a_1^g \dots a_n^g) + (g(f_1' \dots f_n') - (f_1' \dots f_n')) + (g\left(\frac{\partial F}{\partial z_1} \dots \frac{\partial F}{\partial z_n}\right) - \left(\frac{\partial F}{\partial z_1} \dots \frac{\partial F}{\partial z_n}\right))$$
The first two terms are cocycles taking constant values and the last term is exact.
\end{proof}

Next, we will prove the following corollary of proposition (\ref{iso})

\begin{corollary}
\begin{itemize}
\item[1)] Let $X$ be a manifold with the same properties as in proposition \ref{iso}, then all line bundles on $X$ are represented by a cocycle of constants.
\item[2)] The manifold $X$ does not contain any hypersurface.
\end{itemize}
\end{corollary}

\begin{proof}
\begin{itemize}
\item[1)]
Consider the two exponential sequences on $X$
\[\begin{tikzcd}
	0 & {\mathbb{Z}} & {\mathbb{C}} & {\mathbb{C}^{*}} & 0 \\
	0 & {\mathbb{Z}} & {\mathcal{O}_X} & {\mathcal{O}_X^{*}} & 0
	\arrow[from=1-1, to=1-2]
	\arrow[from=1-2, to=1-3]
	\arrow[from=1-3, to=1-4]
	\arrow[from=1-4, to=1-5]
	\arrow[from=2-1, to=2-2]
	\arrow[from=2-2, to=2-3]
	\arrow[from=2-3, to=2-4]
	\arrow[from=2-4, to=2-5]
	\arrow[from=1-2, to=2-2, equal]
	\arrow[from=1-3, to=2-3]
	\arrow[from=1-4, to=2-4]
\end{tikzcd}\]
by going to cohomology, we get:
\[\begin{tikzcd}
	{H^1(X,\mathbb{Z})} & {H^1(X,\mathbb{C})} & {H^1(X,\mathbb{C}^{*})} & {H^2(X,\mathbb{Z})} & {H^2(X,\mathbb{C})} \\
	{H^1(X,\mathbb{Z})} & {H^1(X,\mathcal{O}_X)} & {H^1(X,\mathcal{O}_X^{*})} & {H^2(X,\mathbb{Z})} & {H^2(X,\mathcal{O}_X)}
	\arrow[from=1-1, to=1-2]
	\arrow[from=1-2, to=1-3]
	\arrow[from=1-3, to=1-4]
	\arrow[from=1-4, to=1-5]
	\arrow[from=2-1, to=2-2]
	\arrow[from=2-2, to=2-3]
	\arrow[from=2-3, to=2-4]
	\arrow[from=2-4, to=2-5]
	\arrow[from=1-2, to=2-2]
	\arrow[from=1-3, to=2-3]
	\arrow[from=1-5, to=2-5]
	\arrow[from=1-4, to=2-4, equal]
	\arrow[from=1-1, to=2-1, equal]
\end{tikzcd}\]

As seen in proposition \ref{iso}, the second vertical arrow is an isomorphism, the fifth arrow is injective since $b_2(X)=0$. By the five lemma, the map $H^1(X,\C^*)\to H^1(X,\OO_X^{*})$ is an isomorphism. This means that every line bundle is described by a cocycle of constants.
\item[2)]
Let $L$ be a line bundle and let $\mu \in Hom(\Lambda, \C^{*})\simeq H^1(X, \C^*)$ be the character of $\Lambda$ representing $L$. A holomorphic section of $L$ is a holomorphic function $f\in \Gamma(\OO_G)$ satisfying $g^{*}f=\mu(g)f$, for all $g\in \Lambda$. For any $h\in \Lambda_H$, the order of $h$ in $\Lambda/[\Lambda, \Lambda]$ is finite, so $\mu(h)$ is of finite order in $\C^{*}$. Let $n$ be such $\mu(h)^n=1$ for all $h\in \Lambda_H$, then $f^n$ is invariant under the action of $\Lambda_H$, as shown in the proof of proposition \ref{iso}, this implies that $f^n$ is constant, so $f$ is also constant. Then, if $L$ is not trivial, it does not admit any sections, so there are no hypersurfaces contained in $X$.
\end{itemize}
\end{proof}

\section{Metric aspects}

It is easy to see that the Lie algebra ${\mathfrak g}$ of $G$ has  a natural basis 
$$\{T,A_1\dots A_d, B_1 \dots B_d, C_{ij}\vert i,j \in \{1\dots d\}\}$$
 with the following commutator relations:
$$[A_i,A_j]=0\ \ \ [B_i,B_j]=0\ \ \ [B_j,A_i]=C_{ij}\ \ \ [T,A_i]=A_i\ \ \ [T,B_i]=-B_i$$
and the generators $C_{ij}$ are all central. 
\begin{remark}
Keeping in mind the notation used in \ref{elgen}, the elements of the basis above are given by: $T=\frac{\partial}{\partial \alpha}$, $A_i=\frac{\partial}{\partial a_i}$, $B_j=\frac{\partial}{\partial b_j}$ and $C_{ij}=\frac{\partial}{\partial c_{ij}}$.
\end{remark}
Considering the identification
$$
\begin{pmatrix}
I_d & \rvline & \begin{matrix} b_d \\ \vdots \\ b_1 \end{matrix} & \rvline & 
\begin{matrix}
c_{1d} & \dots & c_{dd}\\
\vdots & \ddots & \vdots\\
c_{11} & \dots & c_{d1}\\
\end{matrix}\\
\hline
0 & \rvline & \alpha & \rvline & 
\begin{matrix}
a_1 & \dots & a_d
\end{matrix}\\
\hline
0 & \rvline & 0 & \rvline& I_d
\end{pmatrix} = \begin{pmatrix}
I_d & \rvline & \begin{matrix} y_{d0} \\ \vdots \\ y_{10} \end{matrix} & \rvline & 
\begin{matrix}
x_{d0} & y_{d1} & x_{d1} & \dots & y_{dm} & x_{dm}\\
\vdots & \vdots & \vdots & \ & \vdots & \vdots\\
x_{10} & y_{11} & x_{11} & \dots & y_{1m} & x_{1m}\\
\end{matrix}\\
\hline
0 & \rvline & w_2 & \rvline & 
\begin{matrix}
w_1 & y_{01} & x_{01} & \dots & y_{0m} & x_{0m}
\end{matrix}\\
\hline
0 & \rvline & 0 & \rvline& I_d
\end{pmatrix}
$$
we may rename the generators of the Lie algebra ${\mathfrak g}$ such that $W_i=\frac{\partial}{\partial w_i}$, $X_{ij}=\frac{\partial}{\partial x_{ij}}$ and $Y_{ij}=\frac{\partial}{\partial y_{ij}}$. \par
In terms of this notation, the left-invariant complex structure is defined by: 
$$\J(W_1)=W_2\text{, }\J(W_2)=-W_1\text{, }\J(X_{ij})=Y_{ij}\text{, }\J(Y_{ij})=-X_{ij}$$
Alternatively, we may write these relations as:
$$\J(A_1)=T \text{, }\J(T)=-A_1$$
$$\J(A_{2s+1})=A_{2s}\text{, }\J(A_{2s})=-A_{2s+1}\text{, for } s \in \{1 \dots \frac{d-1}{2}\}$$
$$\J(C_{1t})=B_t\text{, }\J(B_t)=-C_{1t}\text{, for } t \in \{1 \dots d\}$$
$$\J(C_{2s+1, t})=C_{2s, t} \text{, }\J(C_{2s,t})=-C_{2s+1,t}\text{, for }s \in \{1 \dots \frac{d-1}{2}\}\text{ and } t \in \{1 \dots d\}$$
For the rest of this section, we will use both ways to denote the basis of ${\mathfrak g}$, in order to simplify the computations.
\begin{definition}
Let $(M,h)$ be a complex hermitian manifold, the metric is called locally conformally K\"ahler (LCK) if the associated form $\omega$ satisfies 
\begin{equation}\label{lck}
d\omega=\theta \wedge \omega
\end{equation}
for some closed $1$-form $\theta$, called the Lee form.
\end{definition}
For more details see \cite{OrVe}.

\begin{remark}
In the case $d=1$, there is an LCK metric on manifolds of type $X_{d,\Lambda}$, that is the Tricerri LCK metric on surfaces of type $S^{(+)}_{N,p,q,r;0}$ \cite{Tr}.
\end{remark}

\begin{proposition}
For $d > 1$, there is no LCK metric on the manifolds $X_{\Lambda, d}$. 
\end{proposition}

\begin{proof}
Let $X$ denote the manifold $X_{\Lambda,d}$ with $d>1$. By the same arguments as in \cite{belgun}, if $X$ has an LCK metric, then it induces a left-invariant LCK metric on $G$. Thus, it is enough to show that such structure cannot exist.
Assume $(\omega, \theta)$ is a left-invariant LCK structure on $G$ and let $g$ be the Riemannian metric associated to $\omega$. The following computations are done at the level of the Lie algebra $\frak{g}$. Since $\theta$ is closed, it vanishes on commutators, henceforth $\theta(A_i)=\theta(B_j)=\theta(C_{ij})=0$ for all $i,j.$ We prove that $\theta(T)=0$ as well.
To do this, we apply the LCK identity \ref{lck}
$$(d\omega)(A, B, C)=(\theta\wedge\omega)(A, B, C)$$
for $A=W_2 (=T), B=X_{11}, C=Y_{11}.$ By direct computation, the LHS vanishes, while the RHS equals
$\theta(W_2)\omega(X_{11}, Y_{11})$. But as $Y_{11}=\J(X_{11})$ we get that $\omega(X_{11}, Y_{11})=g(Y_{11}, Y_{11})>0$, which implies $\theta(W_2)=0.$

Hence, $\omega$ is a left-invariant K\"ahler form on $G$. This implies that $X$ is K\"ahler, which is a contradiction
\end{proof}

\begin{definition}
Let $(M,h)$ be a complex hermitian manifold of dimension $n$, the metric is called locally conformally balanced (LCB) if the associated form $\omega$ satisfies 
$$d\omega^{n-1}=\theta\wedge\omega^{n-1}$$
for some closed $1$-form $\theta$.
\end{definition}

\begin{proposition}
The manifolds $X_{\Lambda,d}$ admit a left-invariant LCB metric. 
\end{proposition}

\begin{proof}
Consider the left-invariant form 
$$\omega = W_1^{\vee}\wedge W_2^{\vee} + \sum_{\substack{\substack{\substack 0\leq i \leq 2m+1\\ 0\leq j \leq m\\(i,j)\neq(0,0)}}} X_{ij}^{\vee}\wedge Y_{ij}^{\vee}$$
 This form is positive definite and of type $(1,1)$, so it comes from a hermitian metric $h$. The metric $h$ defined this way is LCB. To prove this, we will first determine $d\omega$ using the fact that $d\omega(A,B,C)=-\omega([A,B],C)-\omega([C,A],B)-\omega([B,C],A)$, for any left-invariant vector fields $A,B,C$. For a general left-invariant $2$-form, we have:
$$d\omega(A_i,B_j,B_k)=-\omega([A_i,B_j],B_k)-\omega([B_k,A_i],B_j)=\omega(C_{ij},B_k)-\omega(C_{ik},B_j)\ \ ;$$
$$d\omega(A_i,A_j,B_k)=-\omega([B_k,A_i],A_j)-\omega([A_j,B_k],A_i)=-\omega(C_{ik},A_j)+\omega(C_{jk},A_i)\ \ ;$$
$$d\omega(T,A_i,A_j)=-\omega([T,A_i],A_j)-\omega([A_j,T],A_i)=2\omega(A_j,A_i)\ \ ;$$
$$d\omega(T,B_i,B_j)=-\omega([T,B_i],B_j)-\omega([B_j,T],B_i)=-2\omega(B_j,B_i)\ \ ;$$
$$d\omega(T,A_i,B_j)=-\omega([T,A_i],B_j)-\omega([B_j,T],A_i)-\omega([A_i,B_j],T)=$$
$$=-\omega(A_i,B_j)-\omega(B_j,A_i)-\omega(-C_{ij},T)=\omega(C_{ij},T)\ \ ;$$
$$d\omega(A_i, B_j, C_{st})=-\omega([A_i,B_j],C_{st})=\omega(C_{ij},C_{st})\ \ ;$$
$$d\omega(T, A_i, C_{st})=-\omega([T,A_i],C_{st})=-\omega(A_i,C_{st})\ \ ;$$
$$d\omega(T, B_j, C_{st})=-\omega([T,B_j],C_{st})=\omega(B_j,C_{st}).$$
All the remaining possible combinations are zero directly by the commutator relations. Applying these calculations to the form $\omega$ defined above, we conclude that 
$$d\omega=\sum_{\substack{1\leq i \leq 2m+1\\ 1\leq j \leq m}} X_{0j}^{\vee}\wedge Y_{i0}^{\vee}\wedge Y_{ij}^{\vee}-\sum_{\substack{1\leq i \leq 2m+1\\ 1\leq j \leq m}} Y_{0j}^{\vee}\wedge Y_{i0}^{\vee}\wedge X_{ij}^{\vee}+\sum_{1\leq j \leq m} W_2^{\vee}\wedge X_{0j}^{\vee}\wedge Y_{0j}^{\vee}$$

 Then, $d\omega^{n-1}=(n-1)d\omega\wedge\omega^{n-2}=(n-1)\left(\displaystyle\sum_{1\leq j \leq m} W_2^{\vee}\wedge X_{0j}^{\vee}\wedge Y_{0j}^{\vee}\right)\wedge \omega^{n-2}$, since $X_{0j}, Y_{i0}$ and $Y_{ij}$ come from three different $(X, Y)$ pairs  (and so do $Y_{0j}, Y_{i0}$ and $X_{ij}$) and thus the first two terms of $d\omega$ vanish when multiplied with $\omega^{n-2}$. Finally, we obtain
$$(n-1)d\omega\wedge\omega^{n-2}=  (n-1)(n-2)!(2m+1) W_2^{\vee}\wedge \bigwedge_{\substack{\substack{\substack 0\leq i \leq 2m+1\\ 0\leq j \leq m\\(i,j)\neq(0,0)}}} X_{ij}^{\vee}\wedge Y_{ij}^{\vee}=$$
$$=(2m+1) W_2^{\vee}\wedge \omega^{n-1}$$ 
The form $\theta=(2m+1) W_2^{\vee}$ is closed because $W_2$ cannot be obtained as the bracket of two other left-invariant vector fields. 
\end{proof}


\begin{thebibliography}{1000}

\bibitem[AbKo]{Abe}
\newblock  Y. Abe, K. Kopfermann, 
\newblock {\em Toroidal groups. Line bundles, cohomology and quasi-Abelian varieties.}
\newblock Lecture Notes in Mathematics. 1759. Berlin: Springer. viii, (2001).

\bibitem[Bel]{belgun}
F.A. Belgun.
\newblock {\em On the metric structure of non-{K}\"{a}hler complex surfaces.}
\newblock  Math. Ann., 317({\bf 1}), (2000) 1--40.

\bibitem[EnPa]{EPaj}
\newblock  H. Endo, A. Pajitnov, 
\newblock {\em On generalized Inoue manifolds}, 
\newblock Proceedings of the International Geometry Center {\bf 13}, 4
(2020), 24--39.



\bibitem[In]{Ino}
\newblock M. Inoue, 
\newblock {\em On surfaces of class ${\rm VII}_0$}, 
\newblock Invent. Math. {\bf 24},  (1974), 269--310.

\bibitem[KaUm]{Kaz}
\newblock  H. Kazama,T.  Umeno, 
\newblock {\em  $\overline{\partial }$ Cohomology of Complex Lie Groups, }
\newblock Publ. Res. Inst. Math. Sci. 26, No. 3,  (1990), 473--484.


\bibitem[Mu]{Mum}
\newblock D. Mumford,
\newblock {\em Abelian varieties,}
 \newblock Studies in Mathematics. Tata Institute of Fundamental Research 5. London: Oxford University Press. VIII, 242 p. (1970).




\bibitem[OeMi]{OeMi}
\newblock   K. Oeljeklaus, C. Miebach,
\newblock {\em Compact complex non-{K}\"{a}hler manifolds associated with totally real reciprocal units,}
\newblock Math. Zeitschrift {\bf 301},  (2022),  2747--2760.



\bibitem[OeTo]{OeTo}
\newblock K. Oeljeklaus, M. Toma, 
\newblock {\em Non-{K}\"{a}hler compact complex manifolds associated to number fields},
\newblock Ann. Inst. Fourier
{\bf 55} 1 (2005), 161--171.


\bibitem[OrVe]{OrVe}
\newblock  L. Ornea, M. Verbitsky,
\newblock {\em Principles of Locally Conformally K\"ahler Geometry},
\newblock Progress in Mathematics 354, Birkh\"auser, (2024).


\bibitem[OtTo]{OtTo}
\newblock  A. Otiman, M. Toma, 
\newblock {\em Hodge decomposition for Cousin groups and for Oeljeklaus-Toma manifolds},
\newblock Ann. Sc. Norm. Super. Pisa, Cl. Sci. ({\bf 5}) 22, No. 2, (2021), 485--503.


\bibitem[Ra]{rag}
\newblock  M.S. Raghunathan,
\newblock {\em Discrete subgroups of Lie groups.}
\newblock Ergebnisse der Mathematik und ihrer Grenzgebiete. Band 68. Berlin-Heidelberg-New York: Springer-Verlag, (1972).

\bibitem[Tr]{Tr}
\newblock F. Tricerri,
\newblock {\em Some examples of locally conformal K\"ahler manifolds},
\newblock Rend. Semin. Mat. Univ. Politecn. Torino {\bf 45} (1987), 117--123.


\bibitem[Wa]{wall}
\newblock  C.T.C. Wall, {\em Geometric structures on compact complex analytic manifolds}, Topology {\bf 25}
(1986), 119--153.




\end{thebibliography}
\end{document}